\documentclass[11pt,reqno]{amsart}
\usepackage{amsfonts}
\usepackage{amsmath}
\usepackage{amssymb}
\usepackage{amsthm}
\usepackage{graphicx}
\usepackage{epsfig}

\numberwithin{equation}{section}

\newtheorem{thm}{Theorem}[section]
\newtheorem{prop}{Proposition}[section]
\newtheorem{lem}{Lemma}[section]
\newtheorem{rem}{Remark}[section]

\newtheorem{defn}{Definition}[section]
\newtheorem{exm}{Example}[section]
\newtheorem{asm}{Assumption}[section]

\newcommand{\thmref}[1]{Theorem~{\rm \ref{#1}}}
\newcommand{\lemref}[1]{Lemma~{\rm \ref{#1}}}

\newcommand{\propref}[1]{Proposition~{\rm \ref{#1}}}

\newcommand{\exmref}[1]{Example~{\rm \ref{#1}}}
\newcommand{\asmref}[1]{Assumption~{\rm \ref{#1}}}

\renewcommand{\P}{\mathbb{P}}
\newcommand{\R}{\mathbb{R}}
\newcommand{\E}{\mathbb{E}}

\def\one{{\hbox{1{\kern -0.35em}1}}}

\title[]{On the Continuity of Stochastic Exit Time Control Problems}
\thanks{This research is supported in part by the National Science
  Foundation under Grant NSF-DMS-0906257.}  
\thanks{We would like to thank the Corresponding Editor Paul Chow and the anonymous referee for their feedback which helped us improve our paper.}

\author[]{Erhan Bayraktar}
\address{Department of Mathematics, University of Michigan, Ann Arbor,
  MI 48109}
\email{erhan@umich.edu}

\author[]{Qingshuo Song}
\address{Department of Mathematics, City University of Hong Kong}
\email{song.qingshuo@cityu.edu.hk}

\author[]{Jie Yang}
\address{Department of Mathematics, Statistics, and
Computer Science, University of Illinois at Chicago, Chicago, IL
60607}
\email{jyang06@math.uic.edu}

\begin{document}

\date{April 16, 2010}

\maketitle

\begin{abstract}

We determine a weaker sufficient condition than that of Theorem 5.2.1
in Fleming and Soner (2006) for the continuity of the value functions
of stochastic exit time control problems.
\end{abstract}

\section{Introduction}

Let $(\Omega, \mathcal{F}, \mathbb{F}=(\mathcal{F}_s)_{t\le s<\infty}, \P)$ be a filtered probability space
satisfying 
the usual conditions and $W$ be an $\mathbb{R}^{d}$ valued Brownian motion adapted to $\mathbb{F}$. Consider the following stochastic differential equation in $\R^n$
\begin{equation}
  \label{eq:sde}
  d X_s = b(s, X_s, \alpha_s) ds + \sigma (s, X_s, \alpha_s) dW_s, \quad
\end{equation}
where $\alpha_t$ the control belongs to $\mathcal{A}$, the set of all progressively measurable processes with values in a compact subset $A$ of $\mathbb{R}^k$.

Let $O\subset \mathbb{R}^n$ be a
bounded open set, and set $Q = [0,T)\times O$. For a given initial
$(t,x) \in Q$, define $\tau$ as the
first exit time of the $\mathbb{R}^{n+1}$-valued process $(s,X_s)$ from
the bounded domain $Q$, that is
\begin{equation}\label{eq:tau}
  \tau = \inf\{s\ge t: (s,X_s) \notin Q\}.
\end{equation}

Given a running cost function $\ell:
\mathbb{R}_{+} \times \mathbb{R}^n \times A \to \mathbb{R}$ and a terminal
cost function $g: \mathbb{R}_{+} \times \mathbb{R}^n \to \mathbb{R}$,
we define the value function as
\begin{equation}\label{eq:valp}
  V(t,x) = \inf_{\alpha \in \mathcal{A}} \mathbb{E}_{t,x}\left\{\int_t^\tau \ell(s, X_s,
  \alpha_s) ds + g(\tau, X(\tau))\right\},
\end{equation}
in which $\mathbb{E}_{t,x}$ is the expectation operator conditional on
$X_t=x$. Occasionally, we will refer to $X$ as $X^{t,x}$ to emphasize
its initial condition.  

In general one can show that the value function is a viscosity solution of a fully non-linear Hamilton-Jacobi-Bellman equation given that it is a continuous function; see Corollary 3.1 on page 209 of \cite{FS06}.
However, when the domain is bounded, it is not always the case that the
value function is continuous  due to {\it tangency problem} mentioned in
\cite[pp. 278-279]{KD01}, which imposes continuity as an additional assumption.
Consider two underlying processes $X^1 = X^{t,x^1}$ (solid
line) and $X^2 = X^{t,x^2}$ (dotted line) in Figure
\ref{fig:png}. No matter how close $X^1$ and $X^2$ are, the 
difference between their first exit time  $\tau_1$ and $\tau_2$ could
be very large. 
\begin{figure}[htb]
  \centering
  \includegraphics[scale=0.25]{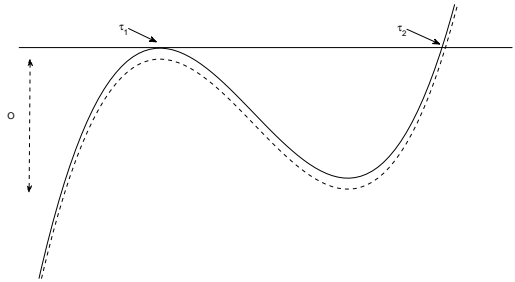}
  \caption{Tangency problem}
  \label{fig:png}
\end{figure}

A sufficient condition for the continuity of the value function is
provided on page 205 of \cite{FS06}.  In this paper we improve this
condition using a probabilistic argument; see \thmref{t-vctn} and
\exmref{e-1d}. We also note that the regularity of the stochastic exit
time control problem has been studied in \cite{LM82}, in which the
value function is shown to be Lipschitz continuous assuming the existence of an appropriate ``global barrier". Under weaker assumptions, similar to the ones considered here,  the
continuity of the value function was obtained by \cite{BB95} and
\cite{IL90} for semi-linear and quasi-linear Dirichlet problems,
respectively, using purely PDE methods.  More recently, the continuity
of viscosity solutions of fully non-linear Dirichlet problems (with
integro-differential terms) is analyzed in \cite{BCI08}. 
Related results can also be found in
\cite{Kov09}, where the Dirichlet problem for the Isaacs Equation is
discussed. 
With respect
to these aforementioned papers our contribution is to give a simple
probabilistic proof of the continuity result for the fully nonlinear Cauchy problems on bounded domains.

The rest of the paper is organized as
follows: In Section 
 \ref{sec:prelim} we recall some preliminary results.  Section \ref{sec:lb}, is
devoted to an important result on the sample path behavior of the state process on the boundary of the domain of the problem. Using the results developed in Section  \ref{sec:lb},
a sufficient condition on the continuity of the value function is derived in Section
\ref{sec:ctn}. Some of the proofs are given in the Appendix.

\section{Preliminaries} \label{sec:prelim}
This section presents definitions and assumptions needed for the setup
of our problem, and collects some relevant classical results.

To proceed, we present the standing assumptions needed for our work. Below we use $|\cdot|$ for the absolute value of a scalar and
$\|\cdot\|$ for the second Euclidean norm.
\begin{asm}
  \label{a-sde} For any $x, x^1, x^2 \in \R^n$, $a \in A$, $t \in
  [0,T]$, functions  $b, \sigma, \ell$, and $g$ satisfy, for some  strictly
  positive constant $K$
\begin{enumerate}
\item
 $   \|b(t,x^1,a) - b(t,x^2,a)\| + \|\sigma(t,x^1,a) - \sigma(t,x^2,a)\|  
   \le K {\|x^1 - x^2\|};$
\item $\|b(t,x,a)\| + \|\sigma(t,x,a)\| \le
  K(1 + \|x\|)$, $\forall (t,x,a) \in [0,T]\times \mathbb{R}^n \times A$;
\item $\ell$ and $g$ are continuous functions;
\item 
 $    |\ell(t,x^1,a) - \ell(t,x^2,a)|  +|g(t,x^1) - g(t,x^2)|
   \le K {\|x^1 - x^2\|}; \; x^1, x^2 \in \R^n, (t,a) \in [0,T] \times A;$
\item $|\ell(t,x,a)| + |g(t,x)| \le K(1 + \|x\|^2 )$.\end{enumerate}
\end{asm}
The first two of our assumptions guarantee that \eqref{eq:sde} has a unique strong solution for a given $\alpha \in \mathcal{A}$.

Next, we present the dynamic programming principle; see e.g. \cite{FS06, MY99}.
\begin{prop}\label{p-dpp}
  For any stopping time $\theta$ with $t\le \theta \le \tau$, 
  \begin{equation}
    \label{eq:dpp}
    V(t,x) = \inf_{\alpha \in \mathcal{A}} \mathbb{E}_{t,x}\left\{\int_t^{\theta}
    \ell(s, X_s, \alpha_s) ds + V(\theta,
    X_{\theta})\right\}.  
  \end{equation}
  \end{prop}

Let  $\forall \varphi\in C^{1,2}(Q)$
\begin{equation*}
    G^a \varphi (t,x) = \varphi_t(t,x) + L_t^a \varphi (x), 
\end{equation*}
and
\begin{equation}\label{eq:lop}
    L_t^a \varphi (x) = b(t,x,a) \cdot D_x \varphi(t,x)
    + \frac 1 2 \hbox{ tr } (\sigma\sigma'(t,x,a) D_x^2 \varphi (t,x)).    
\end{equation}
Using the dynamic programming principle it can be seen that the value function is a solution of  \begin{equation}   \label{eq:ppde}
\begin{split}
  \inf_{a\in A} \{ G^a V(t,x)  + \ell(t,x,a) \} &= 0, \quad (t,x)\in Q,
  \\ V(t,x) &= g(t,x), \quad (t,x) \in \partial^* Q \triangleq [0,T)\times \partial O \cup
 \{T\} \times O,
\end{split}
\end{equation}
in the sense, which we will now describe.

\begin{defn} \label{d-vsol}
Let $u(t,x) = g(t,x)$, $(t,x) \in \partial^* Q $.  (i) It is called a viscosity subsolution of \eqref{eq:ppde}
 if for any $(t_0,x_0; \varphi) \in Q \times C^{2,1}(Q)$ such that $\varphi(t,x) \leq u(t,x)$, $(t,x) \in Q$, and $\varphi(t_0,x_0)=u(t_0,x_0)$ we have that
 \begin{equation*} 
   \inf_{a\in A} \left\{ G^a \varphi(t_0,x_0)  + \ell(t_0,x_0,a) \right\} \geq 0.
\end{equation*}
 (ii) It is called a viscosity supersolution of \eqref{eq:ppde} if for any $(t_0,x_0; \varphi) \in Q \times C^{2,1}(Q)$ such that $\varphi(t,x) \geq u(t,x)$, $(t,x) \in Q$, and $\varphi(t_0,x_0)=u(t_0,x_0)$ we have that
 \begin{equation*} 
  \inf_{a\in A} \left\{ G^a \varphi(t_0,x_0)  + \ell(t_0,x_0,a) \right\} \leq 0.
\end{equation*}  
(iii)  Finally, $u$ is a viscosity solution if it is both a viscosity
  subsolution and a viscosity supersolution. 
\end{defn}

\begin{prop}
  \label{t-exist1}
Suppose $V(t,x) \in C(\bar Q)$ and \asmref{a-sde} hold. Then the value
function $V(t,x)$ is the unique viscosity solution of \eqref{eq:ppde}. 
\end{prop}

A complete proof of \propref{t-exist1} can be found in 
\cite{FS06}. In  Appendix, we provide an alternative proof for the 
existence part. 

The characterization of the value function in Proposition~\ref{t-exist1} assumes that it is continuous.
However the value function  is not 
necessarily continuous if the domain is a bounded set (see Figure
\ref{fig:png} and \exmref{e-1d}). In the next section we give a sufficient condition
that guarantees the
continuity of the value function. This improves on the condition provided in
 Section V.2 of
\cite{FS06}.

\section{Sample Path Behavior on the boundary of
  domain} \label{sec:lb}
In this section, we will discuss the sample path behavior of It\^{o}
process on $[0,T)\times \partial O$, which turns out to be crucial for
the continuity of the value function. 

For a given constant vector $a \in A$, let $Y$ be the unique strong solution of the following stochastic differential equation:
$$
  dY_s = b(s, Y_s, a) ds + \sigma (s,
  Y_s,a) dW_s, \quad Y_t=y.$$

The main result  of this section, which we will state next, derives a 
sufficient condition \eqref{eq:t-rhoge1}, under which the process $Y$
must hit $\bar O^c$ infinitely many times in any small duration, if
it starts on $\partial O$. To 
formulate our result, let us denote the signed distance function by 
\begin{equation}\label{eq:rhoh}
  \hat \rho(y) \triangleq \begin{cases}
    \hbox{dist} (y, \bar O), & y \notin O;\\
    -\hbox{dist}(y, {O^c}), & y \in O.
    \end{cases}
\end{equation}

\begin{prop}
  \label{t-rhoge}
Let $(t,y) \in [0,T) \times \partial O$ and $a \in A$.  Assume that $\partial O \in C^2$ and that
\begin{equation}
  \label{eq:t-rhoge1}
  \max\{L_t^a \hat \rho (y), \|\sigma'(t,y,a) D\hat \rho(y)\|\}>0.
\end{equation}
Then,
\begin{equation} \label{t-rhoge3}
  \inf\{s>t: Y_s \notin \bar O\} = t \quad \P-\text{a.s.}
\end{equation}
\end{prop}

\begin{rem}\label{rem:sdf-sm}
The assumption that  $\partial O \in C^2$ implies that $\hat{\rho} \in C^{2}$ in a neighborhood of $\partial{O}$; see Lemma 14.16 in \cite{MR1814364}. Also see page 78 of \cite{MR1894435} and the references therein.
\end{rem}

Before we present the proof of this proposition, we will need some preparation. 
First, note that \eqref{t-rhoge3} can be
written as the local behavior of a one-dimensional process $\hat \rho(X_s)$:
\begin{equation*}
  \inf\{s>t: \hat \rho(Y_s) >0\} = t \quad  \P-a.s.
\end{equation*}

Next, we will focus on one-dimensional process, which implies that a
non-degenerate continuous local martingale process $M$ 
starting from zero hits $(0,\infty)$ infinitely many times in any
small time period. If $M$ is a standard Brownian motion, the proof 
is given by Blumenthal 0-1 law \cite[Theorem 7.2.6]{Dur05}. However,
because the distribution of $M$ is not explicitly available, we use the representation of $M$ as a time changed Brownian motion.

\begin{lem} \label{t-mtgl}
Let $\hat B(r)$ be a one-dimensional Brownian motion with
respect to $(\Omega, \mathcal{F}, \mathbb{F}, \mathbb{Q})$.  We assume
that $\hat \sigma$ 
is a one-dimensional progressively measurable process with $\int_t^T
\hat\sigma_r^2 dr <\infty$, so that   $M_s = \int_t^s \hat \sigma_r
d \hat B_r$ is a local martingale.   
Furthermore, we assume that $\hat \sigma_s>0$ $\forall s\in [t, T]$
$\mathbb{Q}$-a.s.  Then 
$\tau = \inf \{s>t: M_s>0\}$ satisfies $\tau=t$ $\mathbb{Q}$-a.s.
\end{lem}
\begin{proof}
First, we can extend function $\hat \sigma$ on $[t,T]$ to $[t,\infty)$
by $\hat \sigma(s) = \hat \sigma (T)$ for all $s>T$. Then, the
quadratic variation of $M$ is a strictly increasing function and it satisfies
$$\langle M \rangle_s = \int_t^s \hat \sigma^2 (r) dr \to \infty \hbox{ as
} s \to \infty, \ \mathbb{Q}-a.s.$$
since $\hat \sigma>0$.  For a given positive $s$, define $T(s) \triangleq
\inf\{ r\ge 0: \langle M \rangle(r) >s \}.$ 
The strictly increasing function $T$ satisfies
$ T\left(\langle M \rangle(s)\right) = s.$
The time-changed process $B_s \triangleq M_{T(s)}$
is a $\mathbb{Q}$-Brownian motion under the filtration $\mathcal{G}_s =
\mathcal{F}_{T(s)}$ and $M_s=B_{\langle M \rangle (s)}$; see
e.g. \cite[Theorem 3.4.6]{KS91}. Thus, $\mathbb{Q}$-almost surely, we have
$$\begin{array}{ll} 
\inf\{s: M_s>0\} & 
= \inf\{s: B_{\langle M\rangle(s)}>0\} 
\\& =  \inf\{T \left(\langle M \rangle (s)\right): B_{\langle
  M\rangle(s)}>0\} 
\\& = T\left( \inf\{ \langle M \rangle (s): B_{\langle
  M\rangle(s)}>0\}\right)
\\& = T
(0) = 0. \end{array} $$
The second equality follows from the fact that $\hat\sigma>0$.
The third, on the other hand, follows from the fact that $T$ is increasing.
\end{proof}

We are ready to prove  \propref{t-rhoge}.
\begin{proof}[Proof of \propref{t-rhoge}]
 We will carry out the proof in two steps. \\
  \textbf{(i)} Let us first assume that $\|\sigma'(t,y,a) D\hat \rho(y)\|>0$. 
  Due to the continuity of this function, there exists a stopping time
  $\tau > t$, (which is less than the exit time from the neighborhood
  mentioned in Remark~\ref{rem:sdf-sm}) such that 
  for  $s\in (t,\tau)$ 
  \begin{equation}\label{eq:t-rhoge2}
    \|\sigma'(s,Y_s,a) D\hat{\rho}(Y_s)\| > \varepsilon \triangleq
    \frac 1 2 \|\sigma'(t,y,a) D\hat \rho(y)\| > 0,\quad  \P-a.s.
  \end{equation}

  Thus, applying It\^{o}'s formula, we obtain
  $$\begin{array}{ll}
    \hat \rho (Y_s) & \displaystyle = \int_t^s L_r^a \hat \rho(Y(r))
    dr + \int_t^s D \hat \rho(Y(r)) \sigma(r, Y(r),a) dW(r) \\
    & = \displaystyle \int_t^s L_r^a \hat \rho (Y(r)) dr + \int_t^s
   \|\sigma'(r,Y(r),a) D\hat \rho(Y(r))\| d\widetilde{W}(r) 
  \end{array}$$
  where $\widetilde{W}$ is a one-dimensional $\P$-Brownian motion.
By Girsanov's theorem, there exists  $\mathbb{Q} \sim \P$,
  such that 
  $$\begin{array}{ll}
    \hat \rho (Y_s) \displaystyle = \int_t^s \|\sigma'(r, Y(r),a) D\hat
    \rho(Y(r))\| d \widetilde{W}_r^{\mathcal{Q}}  
  \end{array}$$
  where $\widetilde{W}_r^{\mathcal{Q}}  $ is a $\mathbb{Q}$-Brownian
  motion. Thus, $\hat \rho(Y_s)$ is a local martingale process under
  $\mathbb{Q}$.  \lemref{t-mtgl} implies that
  \begin{equation*}
    \inf\{s>t: \hat \rho(Y_s)>0 \} = t, \quad \mathbb{Q}-\text{a.s.}
  \end{equation*}
  Since $\P$ is equivalent to $\mathbb{Q}$, and the conclusion holds
  $\P$-a.s. \\ 

\noindent \textbf{(ii)} This was a case already proved in \cite[Lemma
V.2.1]{FS06}. 
\end{proof}

\section{Continuity of the value function} \label{sec:ctn}

We will construct a sequence of functions that converge uniformly to
the value function. 
For this purpose let $\hat d(x) = \hat \rho^+(x)$ and define 
$
  \Lambda^\varepsilon(s, X)\triangleq \displaystyle \exp \left\{- \frac 1
  \varepsilon \int_t^s \hat d (X_r) dr \right\}.$
Let 
\begin{equation}
  \label{eq:coste}
  J^\varepsilon (t,x,\alpha) = \mathbb{E}_{t,x}\left \{ \int_t^T
\Lambda^\varepsilon(s, X) \ell(s, X_s, \alpha_s) ds +
\Lambda^\varepsilon(T, X) g(T, X(T))\right\}.
\end{equation}
and
\begin{equation}\label{eq:value-func-eps}
\begin{array}{ll}
  V^\varepsilon(t,x) = \inf_{\alpha\in \mathcal{A}} J^\varepsilon(t,x,\alpha).
\end{array}
\end{equation}
Next, \lemref{l-vecon}  shows the continuity of
this function. Its proof is given in the Appendix.
\begin{lem}
  \label{l-vecon}
Under \asmref{a-sde}, $V^\varepsilon \in
C([0,T]\times \bar O)$. In fact,
$$\begin{array}{ll}
  |V^\varepsilon(t_1,x^1) - V^\varepsilon(t_2,x^2)| \le
  C_{\varepsilon}(\|x^1-x^2\| + |t_1-t_2|^{1/2}),
\end{array}$$
for some positive constant $C_{\varepsilon}$.
\end{lem}

\begin{thm} \label{t-vctn}
Assume that \asmref{a-sde} and the following hold:
 \begin{enumerate}
  \item $\partial O \in C^2$; 
  \item  $\forall (t,x) \in [0,T) \times \partial O$, there exists an
    $a\in A$ satisfying \eqref{eq:t-rhoge1};
          \item 
    \begin{equation}\label{eq:a-conp2}
      \inf_{a\in A} \{G^a (u + g) (t,x) + \ell(t,x, a)\} \ge 0, \forall
      (t,x) \in  [0,T]\times \mathbb{R}^n.
    \end{equation}
  \end{enumerate}
Then $V$ is continuous on $\bar Q$. 
\end{thm}

\begin{rem}
In \cite[Pages 202-203]{FS06}, a sufficient condition for the continuity of the value function is given: $L_t^a \hat \rho (y)>0$ for some $a\in A$ for all $(t, y) \in [0,T)
\times \partial O$.
\thmref{t-vctn} provides an alternative sufficient condition:
$\|\sigma'(t,x,a) D\hat \rho(x)\|>0$ for some $a\in A$ for all $(t,x)
\in [0,T) \times \partial O$.
\end{rem}

\begin{proof} The proof is divided into two steps.

\noindent \textbf{(i)} Assume that $\ell \ge 0, g= 0$ on
$\mathbb{R}_{+}\times \mathbb{R}^n \times A$. 
    Fix $(t,x) \in [0,T)\times \partial O$. Let $a \in A$  satisfy
    \eqref{eq:t-rhoge1}. Consider the constant control process $\{a_s
    \equiv a: s\ge t\}$ and let $Y$ denote the corresponding
    process governed by this constant control. By \thmref{t-rhoge} for
    $s\in (t,T]$ we have 
    $$\begin{array}{ll}
      \displaystyle \int_t^s \hat d(Y_r) dr >0,  \quad \P-\text{a.s.}
    \end{array}$$
    Hence,
    $$\begin{array}{ll}
      \lim_{\varepsilon \to 0^+}\Lambda^\varepsilon(s, Y) = 0 \quad
      \P-a.s.
    \end{array}$$
    By Dominated Convergence Theorem, one can conclude that
    $$\begin{array}{ll}
      \lim_{\varepsilon \to 0^+} \mathbb{E}_{t,x}\left \{ \int_t^T
        \Lambda^\varepsilon(s, Y) \ell(s, Y_s, a) ds +
        \Lambda^\varepsilon(T, Y) g(T, Y_T)\right\}=0.
    \end{array}$$
    This implies 
    $$\lim_{\varepsilon \to 0^+} J^\varepsilon(t,x,a) = 0.$$
    Together with  $J^\varepsilon(t,x,a) \ge V^\varepsilon(t,x) \geq
    0$ which follows from \eqref{eq:value-func-eps}, the above implies
    that 
    \begin{equation}
      \label{eq:rk2}
      \lim_{\varepsilon\to 0^+} V^\varepsilon(t,x) = 0 = V(t,x), \quad
       (t,x)\in [0,T]\times \partial O.
    \end{equation}
    Therefore, $V^\varepsilon(t,x)$ is continuous (\lemref{l-vecon})
    on the compact set $[0,T] \times \partial O$ in
    $\mathbb{R}^{n+1}$, and it monotonically 
    converges to the zero function. Dini's theorem implies that
    $\lim_{\varepsilon \to 0^+} V^\varepsilon (t,x) = 0$ uniformly on
    $[0,T] \times \partial O$. Thanks to the uniform convergence, if we set
    $$\begin{array}{ll}
      h(\varepsilon) \triangleq \sup\{V^\varepsilon(t,x): (t,x)
      \in [0,T] \times \partial O\},
    \end{array}$$
     we have that $\lim_{\varepsilon\to 0^+}h(\varepsilon)  = 0$.
    
    Now we are ready to prove the continuity of the value function
    $V$. Let $(t,x) \in Q$. Applying the dynamic programming
    principle to $V^\varepsilon(\cdot, \cdot)$ with respect to
    stopping time $\tau$ of \eqref{eq:tau}, and using the fact that
    $\Lambda^\varepsilon(s, X_s^{t,x,\alpha}, \alpha_s) \equiv 1$ for
    $s\le \tau$ and $\alpha \in \mathcal{A}$, we obtain 
    \begin{equation}
      \label{eq:rk1}
      \begin{array}{ll}
        V^\varepsilon(t,x) & \displaystyle = \inf_{\alpha\in
          \mathcal{A}} \left \{ \mathbb{E}_{t,x} 
          \left[\int_t^\tau \ell (s, X_s^{t,x,\alpha}, \alpha_s) ds +
            V^\varepsilon (\tau, X^{t,x,\alpha}_{\tau})\right] \right\} \\
        & \le \displaystyle \inf_{\alpha\in \mathcal{A}} \left \{
          \mathbb{E} \left[\int_t^\tau \ell 
            (s, X^{t,x,\alpha}_s, \alpha_s) ds \right]\right\}+
        h(\varepsilon), \quad \hbox{ since } (\tau, X^{t,x,\alpha}_{\tau})
        \in \partial^* Q\\
        & = V(t,x) + h(\varepsilon).
      \end{array}
    \end{equation}
    Since $\ell \ge 0$, we further have that
    $$\begin{array}{ll}
      V(t,x) \le V^\varepsilon(t,x) \le V(t,x) + h(\varepsilon), \
      \forall (t,x) \in \bar Q
    \end{array}$$
    This implies $V^\varepsilon \to V$ uniformly on $\bar Q$.   Since $V^\varepsilon$ is continuous by \lemref{l-vecon}, the value function $V$ is
    also continuous. \\
    \\
\noindent \textbf{(ii)} The proof follows from (i) once we let $\tilde{l}(t,x,a) \triangleq l(t,x,a)+G^{a}g(t,x)$ and consider \eqref{eq:valp} and \eqref{eq:value-func-eps} by setting $l=\tilde{l}$ and $g=0$.
   
\end{proof}

Next, we give an example, whose value function is
continuous, although it does not satisfy the sufficient condition of
\cite{FS06}. In this example, we first consider a deterministic exit
time problem. We observe that this problem does not have a continuous
value function. Next, we consider a degenerate random version of the
same problem. In this problem, the sufficient condition $L_t^a\hat
\rho(x) >0$ of \cite{FS06} holds only for some points $x$ on the
boundary. Yet, it still satisfies the sufficient condition of
\eqref{eq:t-rhoge1} on the entire boundary, and therefore, the value
function is continuous.

\begin{exm} \label{e-1d}
  {\rm
    \noindent \textbf{(i)} Let $X^{t,x}_s$, $s \geq t$, be the
    one-dimensional process satisfying 
    \begin{equation*}
      d X^{t,x}_s = -2(s-1) ds, \ X^{t,x}_t = x.
    \end{equation*}
    Let $Q = [0,2)\times (-1,1)$, and $\tau^{t,x} = \inf\{s>t:
    X^{t,x}(s) \notin (-1,1) \}$. Let us define
    the value function as $V(t,x) = (\tau^{t,x}\wedge 2) - t$. Then,
    $X^{t,x}$ has an 
    explicit form: $$X^{t,x}_s = -(s-1)^2 + x+(t-1)^2.$$ Therefore,
    the function $s\to X_s^{t,x}$ first increases towards its
    maximum $$ 
      \max_{s\ge t} X^{t,x}_s = x + (t-1)^2,$$ and upon reaching it decreases
    to $-\infty$. Thus, if $x+(t-1)^2 \ge 1$, then $X^{t,x}_{\tau^{t,x}} = 1$, otherwise
    $X^{t,x}_{\tau^{t,x}} = -1$. As a result, for $t \in[0,1]$, $V(t,x)$
    is discontinuous at every point on the parabola
    \begin{equation*}
     \left\{(t,x) \in Q: \max_{s\ge t} X_s = 1\right\} =  \left\{(t,x)
       \in Q: x = -t^2 + 2t\right\}. 
    \end{equation*} 
  We also note that, \eqref{eq:t-rhoge1} does not hold, since
  $$\max\{L^a_t \hat \rho(\pm 1), \|\sigma'(t,x,a)\hat \rho(\pm 1)\| \} =0,
  \quad \forall t\in (0,1).$$ \\
  
  \noindent \textbf{(ii)} Next, we consider the following state process, which we obtain by adding a random perturbation to the
  above deterministic process:
  \begin{equation*}
    d X^{t,x}_s = -2(s-1) ds + (2s - X_s^{t,x})^+ dW_s, \ X^{t,x}_t = x.
  \end{equation*}
This equation admits a unique strong solution since the coefficients are Lipschitz continuous. Let us define the value function to be
  $V(t,x)\triangleq \mathbb{E}_{t,x} [(\tau^{t,x}\wedge 2) -
  t]$. Note
    that, $\hat \rho(\cdot)$ of \eqref{eq:rhoh} satisfies
    \begin{equation}
      \label{eq:e-ndg1}
      \hat \rho(x) = (x-1) \one_{\{x\ge 0\}} + (-1-x) \one_{\{x<0\}},
      \ D\hat \rho(x) = \mathop{sgn}(x), \hbox{ and } D^2 \hat \rho(x)
      \equiv 0. 
    \end{equation}
  As a result,  
  $$L_t \hat \rho(1) = -2(t-1)>0 \hbox{ on } t\in (0,1); \quad 
  |\sigma(t,1) D\hat \rho(1)| = (2t-1)^+>0 \hbox{ on } t\in (1/2, 2),$$ 
  and
  $$L_t \hat \rho(-1) = 2(t-1)>0 \hbox{ on } t\in (1, 2); \quad 
  |\sigma(t,1) D\hat \rho(-1)| = (2t+1)^+>0 \hbox{ on } t \in (0, 2).$$   
  Although, the condition $L_t^a\hat
\rho >0$, which is the sufficient condition given by \cite{FS06}---see equation (2.8) on page 202--- 
 fails on the boundary, 
  the
  continuity of the value function follows from 
\thmref{t-vctn}. 
\qed}
\end{exm}

\section{Appendix}\label{ss-exist}

\subsection{ Proof of Proposition~\ref{t-exist1}} 
First, we will develop the following auxiliary result.

\begin{lem}
  \label{l-estop}
For a given $ (t,x) \in Q$, define $$\theta= \inf\{s>t: (s,
X_s) \notin [t,t+h^2) \times B(x,h)\},$$ where $B(x,h)$ is a ball centered at $x$ with radius $h \in (0,1)$.  Then, 
there exists a constant $K$, which does not depend on the control $\alpha$, such that
$$\begin{array}{ll}
  \mathbb{E}_{t,x} [\theta -t] \ge K h^2. 
\end{array}$$
\end{lem}
\begin{proof}
  Let 
  $f(y) = \|y-x\|^2$.  
  Applying It\^{o}'s formula and taking expectations yield
   \begin{equation}\label{eq:l-estop1}
    \mathbb{E}_{t,x}\{ f(X_{\theta}) - f(x) \} = \mathbb{E}_{t,x} \left\{
    \int_t^\theta L_s^{\alpha_s} f(X_s) ds\right\}.
  \end{equation}
  Since $[t,t+1] \times \bar B(x,1)
  \times A$ is compact, by continuity
  $$\begin{array}{ll}
    \displaystyle \sup_{(s,x,a) \in   [t,t+1] \times \bar B(x,1) \times
      A} |L_s^a f(x)| \le K_{t,x}
    <\infty,
  \end{array}$$
  for some constant $K_{t,x}$.
  Since $(s, X_s, \alpha_s) \in [t,t+1] \times \bar
  B(x,1) \times A $ for any $s\in [t,\theta]$ the integrand in \eqref{eq:l-estop1} is 
  bounded above by $ K_{t,x}$. Since $f(x) = 0$, we can write
  \eqref{eq:l-estop1}  as
  \begin{equation*}
    \mathbb{E}_{t,x} \left[\one_{\{\theta = t + h^2\}} f(X_{\theta})\right] + \mathbb{E}_{t,x}
   \left[ \one_{\{\theta < t+h^2\}} h^2 \right] = \mathbb{E}_{t,x} \left [\int_t^\theta
    L^{\alpha_s} f(X_s) ds \right]\le  K_{t,x} \mathbb{E}_{t,x} \left[\theta
    - t\right]. 
  \end{equation*}
  On the other hand,
  $$\begin{array}{ll}
    \mathbb{E}_{t,x} [\theta -t] \ge \mathbb{E}_{t,x} \left[(\theta - t) \one_{\{\theta
      = t + h^2\}}\right] = h^2 \mathbb{E}_{t,x} \left[ \one_{\{\theta = t + h^2\}}\right].
  \end{array}$$
  Adding the last two inequalities, we get
  $$\begin{array}{ll}
    ( K_{t,x} + 1) \mathbb{E}_{t,x} [\theta -t] \ge h^2 +
    \mathbb{E}_{t,x}\left[ \one_{\{\theta = t + h^2\}} f(X_{\theta})\right] \ge h^2 . 
  \end{array}$$
  The result follows by setting $K \triangleq 1/(K_{t,x} + 1)$. 
\end{proof}
Now, we are ready to prove \propref{t-exist1}. 
\begin{proof}[Proof of \propref{t-exist1}] \hfill \\
\noindent \textbf{(i)}
We will first show that $V$ is a subsolution of \eqref{eq:ppde}.
We will prove the assertion by a contradiction argument.
 Let us assume that there $(t,x;\varphi)$ as in Definition~\ref{d-vsol}-(i) such that
  \begin{equation*}
    \ell (t,x,a) + G^a \varphi (t,x) < - \delta,  \end{equation*}
    for some $\delta>0$.
  Then, by continuity of $\ell + G^a \varphi
  $ in $(t,x)$, there exists $h>0$ such that
  \begin{equation*}
    \ell(s,y,a) + G^a \varphi(y,a) < - \frac \delta 2 < 0, \quad \forall
    (s,y) \in [t, t+h^2) \times B(x,h) \subset Q.
  \end{equation*}
  Let $Y$ be the process which can be obtained by applying the control $\alpha \equiv a$ and define
  \begin{equation*}
    \theta = \inf \{ s>t, Y_s \notin B(x,h)\} \wedge (t+h^2).
  \end{equation*}
  By the dynamic programing principle
  \begin{equation*}
    V(t,x) \le \mathbb{E}_{t,x} \left\{ \int_t^\theta \ell(s,Y_s,a) ds +
    V(\theta, Y_{\theta}) \right\}.
  \end{equation*}
  It follows from how $\varphi$ is chosen that
  $$\begin{array}{ll}
    0 & \le \displaystyle \mathbb{E}_{t,x} \left\{ \int_t^\theta
    \ell(s,Y_s,a) ds + \varphi(\theta,Y_{\theta}) - \varphi(t,x) \right
    \}\\ 
    & = \displaystyle \mathbb{E}_{t,x} \left\{ \int_t^\theta [\ell (s,Y_s,a)
    + G^a \varphi(s,Y_s)] ds \right\}  <- \mathbb{E}_{t,x} \left\{ \int_t^\theta \left(  \frac \delta 2 \right)
    ds \right\} <0, 
  \end{array}$$
  which yields a contradiction.\\ \\
  \noindent \textbf{(ii)} We will now show that $V$ is a supersolution of \eqref{eq:ppde}. We will, again, use proof by contradiction. Let us assume that there exists a triplet $(t,x;\varphi)$ as in Definition~\ref{d-vsol}-(ii) such that 
  \begin{equation*}
    \label{eq:t-exist11}
    \inf_{a\in A} \{\ell(t,x,a) + G^a \varphi(t,x)\} = \delta>0, 
  \end{equation*}
  As a function of $(t,x)$, $\ell(t,x,a) + G^a \varphi(t,x)$
  is equicontinuous in $A$, by  \asmref{a-sde}. Therefore,
  \begin{equation*}
    \inf_{a\in A} \{\ell(t,x,a) + G^a \varphi(t,x)\}
  \end{equation*}
  is also continuous in $(t,x)$. So, one can find $h>0$ such that
  \begin{equation*}
    \inf_{a\in A} \{\ell(s,y,a) + G^a \varphi(s,y)\} > \frac \delta 2
    >0, \quad \forall (s,y) \in [t,t+h^2) \times B(x,h).
  \end{equation*}
  Let $\varepsilon = \frac \delta 4 K h^2$, where $K$ is
  the constant in \lemref{l-estop}. Let $\alpha$ be
  $\varepsilon$-optimal control and define
  \begin{equation*}
    \theta = \inf\{ s>t: X_s \notin B(x,h)\} \wedge (t+h^2).
  \end{equation*}
  Then
  $$\begin{array}{ll}
    V(t,x) & \displaystyle \ge \mathbb{E}_{t,x} \left\{\int_t^\tau \ell(s, X_s,
    \alpha_s) ds + g(\tau, X_{\tau})\right\} - \varepsilon \\
    & \displaystyle \ge \mathbb{E}_{t,x} \left\{\int_t^{\theta} \ell(s, X_s,
    \alpha_s) ds + V({\theta}, X_{\theta})\right\} -  \varepsilon,
  \end{array}$$
  In the following, we obtain the desired contradiction:
  $$\begin{array}{ll}
    0 &\displaystyle \ge \mathbb{E}_{t,x} \left\{\int_t^{\theta} \ell
    (s, X_s, \alpha_s) ds + \varphi
    ({\theta}, X_{\theta}) - \varphi(t,x)\right\} -
    \varepsilon, 
    \\  
    &\displaystyle = \mathbb{E}_{t,x}\left\{ \int_t^{\theta}
    [\ell (s, X_s, \alpha_s) +
    G^{\alpha_s}\varphi (s,X_s)] ds \right\} -
    \varepsilon \\ 
    & \displaystyle \ge \mathbb{E}_{t,x} \left\{\int_t^{\theta}
    [\ell (s, X_s, \alpha_s)  +
    G^{\alpha_s}\varphi (s,X_s)] ds \right\} - \frac
    \delta 4 \mathbb{E}_{t,x}[{\theta} -t], \quad \hbox{ by
      \lemref{l-estop}}\\  
    & \displaystyle = \mathbb{E}_{t,x} \left\{\int_t^{\theta}
    \left[\ell (s, X_s, \alpha_s)  +
    G^{\alpha_s}\varphi (s,X_s) - \frac
    \delta 4\right] ds \right\}\\ 
    & \displaystyle\ge \frac \delta 4 \mathbb{E}_{t,x}[{\theta} -t]  >0.
  \end{array}$$
\end{proof}

\subsection{Proof of Lemma~\ref{l-vecon}}
  First, it can be checked that the following inequality holds:
  $$\begin{array}{ll}
    |\hat d(x^1) - \hat d(x^2) | \le \|x^1 - x^2\|, \quad x^1, x^2
    \in \mathbb{R}^n.
  \end{array}$$
 As a result
  \begin{equation*}
  \begin{split}
    |\Lambda^\varepsilon (s, X^1) - \Lambda^\varepsilon (s, X^2)| 
    &= \displaystyle\left|\exp\left\{ - \frac 1 \varepsilon \int_t^s \hat
    d(X^1_r) dr\right\} - \exp \left\{- \frac 1 \varepsilon \int_t^s
    \hat d(X^2_r) dr\right\}\right| \\ 
    &\le\displaystyle \frac 1 \varepsilon \left|\int_t^s \hat d(X^1_r)
    - \hat d(X^2_r) dr \right|  
    \le\displaystyle \frac 1 \varepsilon \int_t^s \|X^1_r - X^2_r\| dr
    \\ &\displaystyle \le \frac 1 \varepsilon (s-t) \sup_{r\in [t,s]}
    \|X^1_r - X^2_r\|. 
  \end{split}
  \end{equation*}
  For $\varphi  = \ell, g$ we have that
  \begin{equation*}
  \begin{array}{ll}
    \displaystyle \mathbb{E}_{t,x} \left\{|\Lambda^\varepsilon (s, X^1)
    \varphi(s,X^1_s) - \Lambda^\varepsilon (s,X^2) \varphi(s,
    X^2_s)|\right \}\\ 
    \displaystyle \le \mathbb{E}_{t,x} \left\{|(\Lambda^\varepsilon(s, X^1) - \Lambda^\varepsilon
    (s, X^2)) \varphi(s, X^1_s)| \right\}+ \mathbb{E}_{t,x} \left\{|\Lambda^\varepsilon
    (s, X^2) (\varphi(s,X^1_s) - \varphi(s, X^2_s))|\right\}\\
    \displaystyle \le \left(\mathbb{E}_{t,x} |\Lambda^\varepsilon(s, X^1) - \Lambda^\varepsilon
    (s, X^2)|^2\right)^{1/2} \left(\mathbb{E}_{t,x}|\varphi(s, X^1_s)|^2\right)^{1/2} + \\
    \hspace{1in}   \left(\mathbb{E}_{t,x} |\Lambda^\varepsilon 
    (s, X^2)|^2\right)^{1/2}  \left( \mathbb{E}_{t,x} |\varphi(s,X^1_s) - \varphi(s,
    X^2_s)|^2\right)^{1/2} \\
    \le \displaystyle \frac{1}{\varepsilon}(s-t) \left(\mathbb{E}_{t,x} (\sup_{r\in [t,T]}
    \|X^1_r - X^2_r\| )^2\right)^{1/2} + K \left(\mathbb{E}_{t,x} |X^1_s - X^1_s
    |^{2}\right)^{1/2}\\ 
    \le C |x^1 - x^2|,
  \end{array}\end{equation*}
  for some positive constant $C$.
In the above derivation, we utilized 
  $$\begin{array}{ll}
    \mathbb{E} \left[\sup_{t\le s \le t_1} \|X^1_s - X^2_s\|^2\right] \le C
    \|x^1 - x^2\|^2, \quad t\le t_1 \le T,
  \end{array}$$
  for another positive constant $C$.
  Now, we are ready to prove the regularity of $V^\varepsilon$ in $x$. For any
  $x^1, x^2 \in O$,
  \begin{equation*}
  \begin{split}
    |V^\varepsilon(t,x^1) - V^\varepsilon (t,x^2)|
      &\le \sup_{\alpha \in \mathcal{A}} \Bigg\{\mathbb{E}_{t,x}\left[ \int_t^T|\Lambda^\varepsilon(s,X^1) \ell (s,
    X^1(s), \alpha_s) - \Lambda^\varepsilon (s, X^2) \ell(s, X^2(s),
    \alpha_s)| ds\right] \\
   & +\mathbb{E}_{t,x}\left[ \left|\Lambda^\varepsilon(T,X^1) g(T,X^1(T)) -
    \Lambda^\varepsilon(T,X^2) g(T,X^2(T)\right|\right]\Bigg\} \\
    \le & C \|x^1 - x^2\|,
 \end{split}
  \end{equation*}
  for some positive constant $C$. Please refer to \cite{Kry80} for the moment inequalities we used above.
  
  Let us prove the regularity of the value function in $t$. For $t_1<t_2$, we can use the dynamic programming principle to write
  \begin{equation*}
  \begin{split}
    |V^\varepsilon(t_1,x) - V^\varepsilon(t_2,x)|
     &\le \sup_{\alpha} \int_{t_1}^{t_2} \mathbb{E}_{t,x}
    |\Lambda^\varepsilon (s,X) \ell (s,X_s,\alpha_s)| ds +
    \sup_\alpha \mathbb{E}_{t,x} 
    |V^\varepsilon (t_2,X(t_2)) - V^\varepsilon (t_2,x)| \\
    & \le C \left[ \sup_\alpha \int_{t_1}^{t_2} \mathbb{E}_{t,x}
    \left(1+\|X_s\|^2 \right) ds + \mathbb{E}_{t,x} \|X_{t_2}-
    x\| \right]\\ 
    &\le C_1 (t_2 - t_1) +C_2 (t_2 - t_1)^{1/2} \le (C_1 T+C_2) (t_2 - t_1)^{1/2},     
  \end{split}\end{equation*}
  in which $C$, $C_1$ and $C_2$ are positive constants. Here, we used the facts that
  \[
  \E \left[\sup_{0 \leq s \leq T} \|X_s\|^2\right]<\infty,
  \]
 and 
 \[
\sup_{\alpha \in \mathcal{A}} \E_{t,x}\left[\left\|X_s-x\right\|\right] \leq C |s-t|^{1/2},
 \]
 for some constant $C$. \hfill $\square$

\bibliographystyle{plain}


\end{document}